\documentclass[11pt, reqno]{amsart}
\usepackage{amscd}
\usepackage{amsmath}
\usepackage{amssymb}
\usepackage{latexsym}
\usepackage{amsthm}
\usepackage{graphicx}
\usepackage{tikz}
\usepackage{tikz-cd}

\usepackage[all]{xy}       % Include XY-pic

    \SelectTips{cm}{10}     % Use the nicer arrowheads

    \everyxy={<2.5em,0em>:} % Sets the scale I like

    \xyoption{web}          % Include the lattice feature, I don't know why%

\usepackage[backref=page]{hyperref}

\definecolor{darkred}{HTML}{993333}

\newcommand{\arxiv}[1]{\href{http://arxiv.org/abs/#1}{\tt arXiv:\nolinkurl{#1}}}
\newcommand{\arXiv}[1]{\href{http://arxiv.org/abs/#1}{\tt arXiv:\nolinkurl{#1}}}

\newtheorem{theorem}{Theorem}[section]
\newtheorem{lemma}[theorem]{Lemma}

\newtheorem{corollary}[theorem]{Corollary}

\theoremstyle{remark}

\numberwithin{equation}{section}

\newcommand{\nc}{\newcommand}

\nc{\flags}{\mathcal{F}}

\nc{\KP}{\operatorname{KP}}

\def\ii{{\bf i}}
\nc{\re}{re}

\def\N{\mathbb{N}}

\def\C{\mathcal{C}}

\def\F{\mathcal{F}}

\def\I{\mathcal{I}}

\def\jj{{\bf j}}

\def\P{\mathcal{P}} % notati
\def\L{\mathcal{L}}
\nc{\co}{\overline{\nabla}}

\def\la{\lambda}

\def\g{\mathfrak{g}}
\def\uk{\underline{k}}

\def\G{\mathcal{G}}
\def\M{\mathcal{M}}

\def\caln{\mathcal{N}}

\def\Hom{\operatorname{Hom}}

\def\End{\operatorname{End}}

\def\Rep{\,\mbox{Rep}\,}

\def\im{\operatorname{im}}

\def\mods{\mbox{-mod}}

\def\dgmods{\mbox{-dgmod}}
\def\id{\mbox{id}}

\newcommand{\map}[2]{\,{:}\,#1\!\longrightarrow\!#2}

\def\T{\mathcal{T}}

\def\inv{^{-1}}

\def\homb{\Hom^\bullet}
\def\endb{\End^\bullet}

\numberwithin{equation}{section}
\def\X{X}

\title[Monoidality of Kato's Reflection Functors]{Monoidality of Kato's Reflection Functors}
\address{}\email{maths@petermc.net}
\author{Peter J McNamara}
\date{\today}

\begin{document}

\begin{abstract}
Kato has constructed reflection functors for KLR algebras which categorify the braid group action on a quantum group by algebra automorphisms.
We prove that these reflection functors are monoidal.
\end{abstract}
\maketitle

\section{Introduction}

Consider a quantised enveloping algebra $U_q(\g)$ where $\g$ is 
a simple Lie algebra of ADE type. It admits algebra automorphisms $T_i$ for each vertex $i$ of the Dynkin diagram. These automorphisms do not preserve the positive part $U_q(\g)^+$, instead there are explicitly given subalgebras $\ker(r_i)$ and $\ker(_ir)$ of $U_q(\g)^+$ which are mapped isomorphically onto each other via $T_i$. For a precise statement, see \cite[Proposition 38.1.6]{lusztigbook}.

The positive part of the quantum group $U_q(\g)^+$ was categorified in terms of KLR (Khovanov-Lauda-Rouquier) algebras in \cite{kl1}. The subalgebras $\ker(r_i)$ and $\ker(_ir)$ are categorified by the module categories of certain quotients $R(\nu)/\langle e_i\rangle$ and $R(\nu)/\langle _ie\rangle$ of the KLR algebras. These categories are Morita equivalent via functors we call Kato's reflection functors. These were discovered in \cite{kato}, and extended to positive characteristic KLR algebras in \cite{geometric}.

 In this paper, we show that Kato's reflection functors are monoidal. This answers a question from \cite{kkop}, allowing a proof of \cite[Conjecture 5.5]{kkop}. It also fixes an error in \cite[Lemma 4.2(2)]{kato} (in the published version).
Our approach is geometric and the key ingredient is the formality of KLR algebras, which we deduce from their purity. Thus we are necessarily restricted to KLR algebras in characteristic zero.

We make no attempt to discuss the situation beyond finite type. For this, the reader is encouraged to view \cite{kato2}.

This paper was produced independently of \cite{kato2}, which answers the same questions, and the author thanks Kato for providing him with access to a draft of his preprint.

\section{Definitions}

Let $I$ be the index set for a finite type ADE Dynkin diagram.
Fix $i\in I$. Let $Q$ be an orientation of the Dynkin diagram such that $i$ is a source. Let $Q'$ be the quiver obtained from $Q$ by reversing the directions of all arrows incident to $i$.

For $\la\in\N I$, 
let $X_\la$ be the moduli stack of representations of $Q$ of dimension vector $\la$. Let $F_\la$ be the moduli stack of representations of $Q$ of dimension vector $\la$ together with a full flag of subrepresentations. Let $\pi\map{F_\la}{X_\la}$ be the canonical projection. $\pi$ is proper. 

Let $k$ be a field of characteristic zero and define
\[
 \L_\la=\pi_!\uk_{F_\la}[\dim F_\la]
\]
The KLR algebra is defined by
\[
 R(\la)=\Hom^\bullet_{D^b(X_\la;k)}(\L_\la,\L_\la).
\]

Let $X'_\la$, $F_\la'$, $\L_\la'$ and $R(\la)'$ be the corresponding objects defined with $Q'$ in place of $Q$. By the main result of \cite{vv} and the discussion in \cite{kl2} there is an isomorphism $R(\la)\cong R(\la)'$ of graded associative algebras.

Normally in the literature, $R(\la)$ is considered as a graded associative algebra. The category $D^b(X_\la;k)$ has a differential graded enhancement, so we can consider $R(\la)$ as a differential graded algebra. Theorem \ref{formal} below shows that we lose no information by considering the associative algebra $R(\la)$ (which we sometimes consider as a differential graded algebra with trivial differential in order to talk about the category $R(\la)\dgmods$ of differential graded modules over $R(\la)$).

\begin{theorem}\cite[Proposition 10.6]{lusztig}\label{purity}
Suppose $\X_\la$ is defined over a finite field and we take the $l$-adic derived category. Then
 $\L_\la$ is pointwise pure.
\end{theorem}

The following key result is also \cite[Lemma 4.8]{webster}.

\begin{theorem}\label{formal}
 $R(\la)$ is formal.
\end{theorem}

\begin{proof}
 By Theorem \ref{purity}, the cohomology of $R(\la)$ is pure of weight zero, when $\L_\la$ is spread out to a finite field and the $l$-adic derived category is considered. Then \cite[Theorem A.1.1]{pb} shows that this Frobenius action can be lifted to the dg-algebra representing $R(\la)$, and \cite[Proposition 4]{schnurer} implies that $R(\la)$ is formal.
\end{proof}

\begin{theorem}\cite{kato,klr1}
 $R(\la)$ has finite global dimension.
\end{theorem}

\begin{corollary}\label{triequiv}
 There is an equivalence of triangulated categories
 \[
   \langle \L_\la \rangle \cong R(\la)\dgmods
 \]
given by $\Hom^\bullet(\L_\la,-)$, where $\langle L_\la \rangle$ is the full triangulated subcategory of $D^b(X_\la)$ generated by $\L_\la$.
\end{corollary}

\begin{proof}
 Since $R(\la)$ is formal, the image of $\Hom^\bullet(\L_\la,-)$ lands in $R(\la)\dgmods$, where $R(\la)$ has trivial differential. A standard devissage argument shows that $\Hom^\bullet(\L_\la,-)$ induces an equivalence between $\langle \L_\la \rangle$ and the full subcategory of $R(\la)\dgmods$ generated by $R(\la)$. Since $R(\la)$ has finite global dimension, this latter category is all of $R(\la)\dgmods$.
\end{proof}

Let $S_i$ be the simple representation of $Q$ at the vertex $i$. Let $U_\la\subset X_\la$ be the substack of representations $M$ of $Q$ with $\Hom(S_i,M)=0$. Write $j\map{U_\la}{X_\la}$ for the inclusion. Then $j$ is an open immersion.

When considering the quiver $Q'$, we instead define $U_\la'\subset X_\la'$ to be the substack of representations $M'$ of $Q'$ such that $\Hom(M',S_i')=0$. 

The algebra $R(\nu)$ has distinguished idempotents $e_i$ and $_ie$ for each $i\in I$, used to define the relevant categories for Kato reflection functors as in \cite{geometric}.
The category $\C_i(\nu)$ is defined to be the full subcategory of $R(\nu)\mods$ consisting of objects $M$ such that $e_iM=0$. It is thus equivalent to modules over the quotient $R(\nu)/\langle e_i \rangle$.
The category $_i\C(\nu)$ is similarly defined using the idempotent $_ie$. Kato's reflection functors give an equivalence $\C_i(\nu)\cong {_i\C}(s_i\nu)$, where $s_i$ is the simple reflection associated to $i$.

There is an isomorphism \cite{geometric}
\[
 R(\la)/\langle e_i \rangle \cong \homb (j^*\L_\la,j^*\L_\la)
\]

The algebra $\homb (j^*\L_\la,j^*\L_\la)$ is also formal and of finite global dimension. The formality follows from the same purity argument as for $R(\la)$, while the finitude of global dimension is in \cite{kato} and \cite{geometric}.

We can then upgrade Corollary \ref{triequiv} to obtain an equivalence of triangulated categories
\begin{equation}\label{opencorollary}
 \langle j^* \L_\la \rangle \cong R(\la)/\langle e_s \rangle \dgmods
\end{equation}
compatible with the equivalence of the corollary via $j_*$ and the inclusion.

%%%%%%%%%%%%%%%%%%%%%%%%%%%%%%%%%%%%%%%%%%%%%%%%%%%%%%%%%%%%%%%%%%%%%%%%%%%%%%%%%
\section{Comparison of algebraic and geometric induction}
%%%%%%%%%%%%%%%%%%%%%%%%%%%%%%%%%%%%%%%%%%%%%%%%%%%%%%%%%%%%%%%%%%%%%%%%%%%%%%%%%%%%%%%%%

Let $S_{\la\mu}$ be the moduli stack of short exact sequences of representations of $Q$
\begin{equation}\label{ses}
 0\to M'\to M\to M'' \to 0
\end{equation}
where $\dim M'=\la$ and $\dim M''=\mu$.

Let $p\map{S_{\la\mu}}{\X_{\la+\mu}}$ be the map sending the short exact sequence (\ref{ses}) to $M$. Let $q\map{S_{\la\mu}}{\X_\la\times \X_\mu}$ be the map sending (\ref{ses}) to $(M',M'')$. The map $p$ is proper and $q$ is smooth.

The geometric induction functor $\I_G\map{D(X_\la)\times D(X_\mu)}{D(X_{\la+\mu})}$ is
\[
 I_G(\F,\G)=p_!q^*(F\boxtimes \G)[\dim q]
\]

The stack $F_\la$ is the disjoint union of $F_\la^\ii$, where $\ii$ runs over all sequences $\ii=(i_1,\ldots,i_n)$ with each $i_j\in I$ and $\sum_j i_j = \la$. The sequence $\ii$ records the sequence of simple subquotients in the full flag. Let $\P_\ii=\pi_! \uk_{F_\la^\ii} [\dim F_\la^\ii]$. Then 
\(
 \L_{\la}=\oplus_{\ii} \P_\ii.
\)

It is not difficult to check that
\[
 I_G(\P_\ii,\P_\jj)=P_{\ii\jj}
\]
where $\ii\jj$ is the concatenation of the two sequences.

Therefore $I_G(\L_\la,\L_\mu)$ is a direct summand of $\L_{\la+\mu}$. Let $e_{\la\mu}\in R(\la+\mu)$ be the projection to this direct summand.

The algebraic induction functor $I_A\map{R(\la)\mods\times R(\mu)\mods}{R(\la\mu)\mods}$ is
\[
 I_A(M,N)=R(\la+\mu)e_{\la\mu}\bigotimes_{R(\la)\otimes R(\mu)} M\boxtimes N.
\]

\begin{theorem}\label{induction}
 For $\F\in \langle \L_\la\rangle$ and $\G\in \langle \L_\mu \rangle $, there is a natural isomorphism of $R(\la+\mu)$-dg-modules.
 \[
  I_A(\homb(\L_\la\boxtimes \L_\mu,\F\boxtimes \G))\cong \homb (\L_{\la+\mu},I_G(\F \boxtimes \G)).
 \]
\end{theorem}

\begin{proof}
 Consider the two functors
 \begin{align*}
  F&= I_A(\homb(\L_\la\boxtimes \L_\mu,-\boxtimes -)) \\
  G&=\homb (\L_{\la+\mu},I_G(- \boxtimes -)).
 \end{align*}
We begin by constructing a natural transformation $\pi\map{F}{G}$. To construct it, it suffices to find a natural bilinear map
\[
 R(\la+\mu)e_{\la\mu} \times \homb (\L_\la\boxtimes \L_\mu,\F) \to \homb (\L_{\la+\mu}, p_!q^* \F).
\]
This map is
\[
 (xe_{\la\mu},y)\mapsto xe_{\la\mu} p_!q^*(y),
\]
noting that $R(\la+\mu)=\End^\bullet(\L_{\la+\mu})$ and $e_{\la\mu}$ is the projection from $\L_{\la+\mu}$ to $p_!q^*(\L_\la\boxtimes \L_\mu)$.

Now note that $\pi$ is an isomorphism when $\F=\L_\la$ and $\G=\L_\mu$. Then by a standard devissage argument, $\pi$ is an isomorphism whenever $\F$ and $\G$ are in the triangulated categories generated by $\L_\la$ and $\L_\mu$ respectively, as required.
\end{proof}

%%%%%%%%%%%%%%%%%%%%%%%%%%%%%%%%%%%%%%%%%%%%
\section{The reflection functor}
%%%%%%%%%%%%%%%%%%%%%%%%%%%%%%%%%%%%%%%%%%%%

\begin{lemma}\label{lem:jstar}
 Suppose $M\in\C_s$. Then $M\cong \homb(\L,j_*\M)$ for some $\M\in D^b(U_\la;k)$.
\end{lemma}

\begin{proof}
 $M$ is a module over $\homb(j^*\L,j^*\L)$, hence by (\ref{opencorollary}) is of the form $\homb(j^*\L,\M)$ for some $\M$. Since $j_*$ is right adjoint to $j^*$, we get the desired result.
\end{proof}

Let $V_{\la\mu}$ be the moduli stack of short exact sequences of representations of $Q$
\[
 0\to M'\to M\to M''\to 0
\]
where $\dim M'=\la$, $\dim M''=\mu$ and $\Hom(S_i,M')=\Hom(S_i,M'')=0$. Let $V'_{\la\mu}$ be the corresponding moduli stack for the quiver $Q'$.

In \cite{bgp}, reflection functors between the categories $\Rep(Q')$ and $\Rep(Q)$ are constructed which are shown to have the following property:

\begin{theorem}\label{bgg}
 The BGP reflection functor from $\Rep(Q')$ to $\Rep(Q)$ induces isomorphisms $x_{\la}$ and $\tilde x_{\la\mu}$ of stacks such that the following diagram commutes:
 \[
  \begin{CD}
   U_{\la+\mu} @<<< V_{\la\mu} @>>> U_\la\times U_\mu \\
   @Ax_{\la+\mu}AA @AA\tilde x_{\la\mu}A @AAx_\la\times x_\mu A \\
   U'_{s_i\la+s_i\mu} @<<< V'_{s_i\la,s_i\mu} @>>> U'_{s_i\la}\times U'_{s_i\mu}.
  \end{CD}
 \]
\end{theorem}

Under the isomorphism $x\map{U_{s_i\la}'}{U_\la}$, the sheaves $j^*\L_{s_i\la}'$ and $j^*\L_\la$ have isomorphic direct summands (up to shifts). This is because they are semisimple and every simple perverse sheaf on $X_\la$ occurs as a direct summand of $\L_\la$. Therefore there is a Morita equivalence between $\endb(j^*\L_{s_i\la}')$ and $\endb(j^*\L_\la)$.
This Morita equivalence is Kato's reflection functor $\T_i\map{\endb(j^*\L_\la)\mods}{\endb(j^*\L_{{s_i}\la}')\mods}$.

Kato's reflection functor satisfies
\[
 \T_i(\homb(\L_\la,j_*\F))=\homb(\L_{{s_i}\la}',j_*x_\la^*\F)
\]

\begin{lemma}\label{rewrite}
 Consider the following diagram, in which the leftmost and rightmost squares are pullback squares, and the middle square is commutative
  \[  \begin{CD}
   X@<f<< S @= S @>g>> Y\\
   @AjAA   @AjAA @AAhA @AAhA \\
   U @<f<< A @>e>> B @>g>> V.
  \end{CD}
 \]
 Suppose that $f$ is smooth, $g$ is proper and $h$ is an immersion. Then we have the equality of functors from $D^b(U;k)$ to $D^b(V;k)$:
 \[
  h^*g_*f^*j_*\cong (g\circ e)_*f^*
  \]
\end{lemma}

\begin{proof}
 This is a routine consequence of base change and the identity $h^*h_*=\id$.
\end{proof}

\begin{lemma}\label{essimage}
 Suppose $\F\in D^b(U_\la\times U_\mu;k)$. Then
 $p_!q^*j_*\F$ is in the essential image of $j_*\map{D^b(U_{\la+\mu};k)}{D^b(X_{\la+\mu};k)}$.
\end{lemma}

\begin{proof}
Let $Z$ be the complement of $U_{\la+\mu}$ in $X_{\la+\mu}$ and $i\map{Z}{X_{\la+\mu}}$ be the inclusion. By considering the exact triangle $i_!i^!\to \id \to j_*j^*\xrightarrow{+1}$, it suffices to show that $i^!p_!q^*j_*=0$.
Let $S_Z=q\inv(Z)$ and $f\map{S_Z}{X_\la\times X_\mu}$ be the restriction of $p$ to $S_Z$. By base change, since $p$ is proper and $q$ is smooth, $i^!p_!q^*j_*=q_!f^!j_*[\dim q]$. Thus it suffices to show that $f^!j_*=0$, i.e. that $f$ and $j$ have disjoint image in $X_\la\times X_\mu$.

If $(M',M'')\in \im f$ then there exists a short exact sequence $0\to M'\to M\to M''\to 0$ with $\Hom(S_i,M)\neq 0$. Thus either $\Hom(S_i,M')\neq 0$ or $\Hom(S_i,M'')\neq 0$, so either way $(M',M'')\notin \im j$, completing the proof.
\end{proof}

%%%%%%%%%%%%%%%%%%%%%%%%%%%%%%%%%%%
\section{monoidality}
%%%%%%%%%%%%%%%%%%%%%%%%%%%%%%%%%

In this section, we use the following diagram (c.f. Lemma \ref{rewrite}):

\[
 \begin{CD}
  X_{\la}\times X_\mu @<q<<S_{\la+\mu}@= S_{\la+\mu} @>p>> X_{\la+\mu} \\
   @AjAA   @AjAA @AAhA @AAhA \\
    U_\la\times U_\mu @<q<< V_{\la+\mu} @>e>> B @>p>> U_{\la+\mu}
 \end{CD}
\]

Here $B$ is defined so that the right hand square is a pullback square. One easily checks that the left hand side is a pullback square.

\begin{theorem}
 The functor $\T_i$ is monoidal.
\end{theorem}

\begin{proof}
 We have to show that $\T_i(M\circ N)$ is naturally isomorphic to $\T_i(M)\circ \T_i(N)$. Use Lemma \ref{lem:jstar} to write $M$ and $N$ as $\homb(\L,j_*\M)$ and $\homb(\L,j_*\caln)$ respectively.
Then by Theorem \ref{induction}, 
\[
 M\circ N \cong \homb(\L_{\la+\mu},p_!q^*(j_*\M\boxtimes j_*\caln)).
\]
By Lemma \ref{essimage}, we have $p_!q^*(j_*\M\boxtimes j_*\caln)\cong h_*h^*p_!q^*(j_*\M\boxtimes j_*\caln)$. By Lemma \ref{rewrite}, we can write this as
\[
 M\circ N\cong \homb(\L_{\la+\mu}, h_*(p\circ e)_*q^*(j_*\M\boxtimes j_*\caln)).
\]
From the description of Kato's reflection functor, we have
\begin{equation}\label{tsmcn}
 \T_i(M\circ N)\cong \homb(\L',h_*x^*(p\circ e)_*q^*(j_*\M\boxtimes j_*\caln)).
\end{equation}

On the other hand, from the description of Kato's reflection functor,
\[
 \T_i(M)\circ \T_i(N)\cong \homb(\L',j_*x^*\M)\circ \homb(\L,j_*x^*\caln))
\]
Then Theorem \ref{induction} tells us that
\[
 \T_i(M)\circ \T_i(N)\cong \homb(\L',p_!q^*(j_*x^*\M\boxtimes j_*x^*\caln))
\]
Then Lemmas \ref{essimage} and \ref{rewrite} tell us that we have
\begin{equation}\label{tsmtsn}
 \T_i(M)\circ \T_i(N)\cong \homb(\L',h_*(p\circ e)_*q^*(j_*x^*\M\boxtimes j_*x^*\caln)).
\end{equation}
Tracing through all the maps, Theorem \ref{bgg} allows us to identify the right hand sides of (\ref{tsmcn}) and (\ref{tsmtsn}), completing the proof.
% Now we note that the diagram involving $(g\circ e)$ and $f$ can be expressed purely in terms of the subcategories $\Rep(Q)_s$ and $_s\Rep(Q')$ transported to each other under $T_i$ to complete the proof. (we look at the stack of short exact sequences where all terms lie in $\Rep(Q)_s$ (respectively $_s\Rep(Q')$).
 \end{proof}
% 
% \section{beyond finite type}
% 
% A geometric realisation of Lusztig's symmetries in all symmetric types has been given in \cite{xiaozhao}. The analogue of the other necessary results from \cite{geometric} to turn this into a functor between module categories for KLR algebras is in \cite{kato2}. This consists of proving two statements:
% 
% \begin{theorem}
%  $$R(\la)/\langle e_s\rangle \cong \homb(j_*\L_\la,j_*\L_\la)$$
% \end{theorem}
% 
% The proof is in \cite{kato} so we only sketch the approaach here. The key is to study the stratification of $X_\la$ by subvarieties where $\dim\Hom(S_i,M)$ is fixed.
% 
% The results here show that this functor is monoidal when restricted to the full subcategories of modules with finite projective resolutions.

\bibliographystyle{alpha}
\def\cprime{$'$}

\end{document}